\documentclass[14pt]{amsart}
\usepackage{latexsym, amsmath, amscd, amssymb, amsthm}
\usepackage[T2A]{fontenc}
 \usepackage[cp1251]{inputenc}
\usepackage[english]{babel}
 \newtheorem{lm}{Lemma}[section]
  \newtheorem{te}{Theorem}[section]
\begin{document}

\setcounter{page}{1}

\title[Asymptotic behaviour]{The degree of the algebra of covariants of a binary form}
\author{Leonid Bedratyuk  and Nadia Ilash}
\address{Department of Applied Mathematics \\
                Khmelnytskyi National  University\\
                Khmelnytskyi, Instytutska ,11\\
                29016, Ukraine}
\email{leonid.uk@gmail.com}
%\thanks{Research supported in part by the Natural Sciences and Engineering Research Council of Canada and by Emperor Frederick II of Sicily.}

\begin{abstract}
We calculate the degree of the algebra of covariants  $\mathcal{C}_d$ for binary $d$-form. Also, for the degree we obtain its   integral representation and  asymptotic behavior.
\end{abstract}

\maketitle

\section{Introduction}

Let  $R= R_0+R_1+\cdots $  be a finitely generated graded complex algebra, $R_0=\mathbb{C}.$ Denote  by
$$
\mathcal{P}(R,z)=\sum_{j=0}^\infty  \dim R_j z^j,
$$
its  Poincare series.
The number			
$$
\deg(R):=\lim_{z \to 1} (1-z)^r \mathcal{P}(R,z),
$$ 
is called the  degree of the algebra $R$. Here $r$ is  the transcendence degree  of the  quotient field of $R$ over $\mathbb{C}.$ The first two terms of  the Laurent series expansion of $\mathcal{P}(R,z)$  at  the point $z=1$ have  the following form 
$$
\mathcal{P}(R,z)=\frac{\deg(R)}{(1-z)^r}+\frac{\psi(R)}{(1-z)^{r-1}}+ \cdots
$$ 
%????? $r$  ????????????  ??????? ????? ??????????? ???????????? ????????? ??????? $A$, $\dim R_n =\mathcal{O} \left(n^{r-1} \right)$, ???.  \cite{Ben} i ? ???????? ??????????? ??????????????? ??????? $A.$
The numbers $\deg(R), \psi(R)$ are important characteristics of the   algebra $R.$ For instance, if
 $R$ is an  algebra of invariants of a finite group  $G$ then  $\deg(R)^{-1}$ is order of the group $G$ and  $2 \dfrac{\psi(R)}{\deg(R)}$ is the number of the   pseudo-reflections in $G,$  see \cite{Ben}.

Let $V_d$  be standard $ d +1 $-dimensional complex representation of $SL_2$ and let   
 $\mathcal{I}_d:=\mathbb{C}[V_d]^{SL_2}$ be the corresponding   algebra of invariants. In the language of the classical invariant theory the algebra  $I_d$ is called the algebra of invariants  for $d$-binary form of degree $d.$ %????? ??????, ?? ??????? ????????????????? ???? ?????? ??????? $\mathcal{I}_d$ ??? ????? $\mathbb{C}$ ???????? $d-2$.
% ?????? ????????, ?? ??????? ???? ?????? ??? $\mathcal{P}(\mathcal{I}_d,z)$ ? ????????? ????? $z=1$ ??? ?????? 
%$$
%\mathcal{P}(\mathcal{I}_d,z)=\frac{i_0}{(1-z)^{d-2}}+\frac{i_1}{(1-z)^{d-3}}+\cdots
%$$
The following explicit formula for the degree $\deg(\mathcal{I}_d)$ was derived by Hilbert  in  \cite{Hilb}:
$$
\deg(\mathcal{I}_d)=\left\{  \begin{array}{l} \displaystyle 
-\dfrac{1}{4d!}\sum_{0 \leq e  < d/2}(-1)^e { d \choose e} \left(\frac{d}{2}-e \right)^{d-3}, \text{ if $d$  is odd},\\
\displaystyle -\dfrac{1}{2d!}\sum_{0 \leq e  < d/2}(-1)^e { d \choose e} \left(\frac{d}{2}-e \right)^{d-3}, \text{if $d$ is even}.
\end{array}
  \right.
$$
In \cite{SPB} and  \cite{SP} Springer obtained   two different proofs of this result. Also, he  had  found  an integral representation and asymptotic behavior for Hilbert's constants. For this purpose Springer used an  explicit formula for computation of the Poincare series $\mathcal{P}(\mathcal{I}_d,z)$  derived in his paper \cite{SP}.

Let $\mathcal{C}_d$ be  the algebra of  the covariants  if the binary $d$-form, i.e. $\mathcal{C}_d \cong \mathbb{C}[V_1 \oplus V_d]^{SL_2}.$ In the present  paper we calculate  $ \deg(\mathcal{C}_d)$ and   $\psi(\mathcal{C}_d)$.  Also, we calculate   an   integral representation and  asymptotic behavior of the constants. For this purpose we use the explicit formula for the Poincare series $\mathcal{P}(\mathcal{C}_d,z)$   derived by the first author in 
 \cite{BB}. 

\section{Computation of the $\deg(\mathcal{C}_d)$.}

The algebra of covariants $\mathcal{C}_d$    is  a finitely generated graded algebra
$$
\mathcal{C}_d=(\mathcal{C}_d)_0+(\mathcal{C}_d)_1+\cdots+(\mathcal{C}_d)_i+ \cdots,
$$
where each subspace $(\mathcal{C}_d)_i$ of covariants of the degree $i$ is finite-dimensional, $(\mathcal{C}_d)_0\cong \mathbb{C}.$ The formal power series 
$$
\mathcal{P}(\mathcal{C}_d,z)=\sum_{i=0}^{\infty }\dim((C_{d})_i) z^i,
$$ 
is called the Poincare series of the algebra of covariants $C_{d}.$ The finitely generation of  $C_{d}$ implies that its Poincare series is the power series expansion of a rational function.

The  following theorem show an explicit form for this rational function. Let  $\varphi_n,$ $n \in \mathbb{N}$ be the linear operator transforms a rational function $f$ in $z$  to the rational function  $ \varphi_n(f)$ determined on $z^n$ by 
$$
\left( \varphi_n(f) \right)(z^n)=\frac{1}{n}\sum_{j=0}^{n-1}f(\zeta_n^j z), \zeta_n=e^{\frac{2\pi i}{n}}.
$$

\begin{te}[\cite{BB}] The Poincare series  $\mathcal{P}(\mathcal{C}_d,z)$ has  the following form 
$$
\begin{array}{l}
\displaystyle \mathcal{P}(\mathcal{C}_d,z)=\sum_{0\leq j <d/2} \varphi_{d-2j} \left( \frac{(-1)^j z^{j(j+1)} (1+z)}{(z^2,z^2)_j\,(z^2,z^2)_{d-j}} \right),
\end{array}
$$

here  $(a,q)_n=(1-a) (1-a\,q)\cdots (1-a\,q^{n-1})$ --- $q$ is q-shifted factorial. 

\end{te}
It is well known that transcendence degree over $\mathbb{C}$ of the quotient field  for the algebra of covariants $\mathcal{C}_d$ coincides with  the order of pole $z=1$ for the rational function $\mathcal{P}(\mathcal{C}_d,z)$ and equals $d.$ Therefore, the first terms  of the Laurent series 
 for $\mathcal{P}(\mathcal{C}_d,z)$ at the point $z=1$ are 
$$
\mathcal{P}(\mathcal{C}_d,z)=\frac{\deg(\mathcal{C}_d)}{(1-z)^{d}}+\frac{\psi(\mathcal{C}_d)}{(1-z)^{d-1}}+\cdots
$$

In order to calculate the rational coefficients $\deg(\mathcal{C}_d), \psi(\mathcal{C}_d)$ we shall prove several auxiliary facts.  

\begin{lm} The following statements hold:
$$
\begin{array}{ll}
(i) & \text{ the  first terms   of  the Taylors series for  the function $(z^2,z^2)_j$  at  $z=1$ are}\\

& (z^2,z^2)_j=2^j j! (1-z)^j-2^{j-1} j!\, j^2 (1-z)^{j+1}+\cdots;\\

(ii) & \text{ the first terms  of the Laurent series for the function
$\dfrac{(-1)^j z^{j(j+1)} (1+z)}{(z^2,z^2)_j\,(z^2,z^2)_{d-j}}$ at  $z=1$ are}\\ \\
& \displaystyle \frac{(-1)^j}{2^{d-1} j! (d-j)!}\frac{1}{(1-z)^{d}}+\frac{(-1)^j}{2^{d-1} j! (d-j)!}(d+1)\left(\frac{1}{2}d-j-\frac 12 \right)\frac{1}{(1-z)^{d-1}}+\cdots 
\end{array}
$$
\end{lm}
\begin{proof}
$(i)$  We have
$$
(z^2,z^2)_j=(1-z^2)(1-z^4) \ldots (1-z^{2j}).
$$
Let us expand the polynomial $ 1-z ^ n$ in the  Taylor series about $ 1-z.$ We have
$$
1-z^n=-n (z-1)-\frac{n(n-1)}{2!}(z-1)^2+\cdots=n(1-z)-\frac{n(n-1)}{2!}(1-z)^2+O((1-z)^{3})
$$
Therefore
\begin{align*}
&(z^2,z^2)_j=(1-z^2)(1-z^4) \ldots (1-z^{2j})=\\&=(2(1-z)-\frac{2}{2!}(1-z)^2+\cdots)(4(1{-}z)-\frac{4\cdot 3}{2!}(1{-}z)^2+\cdots)\ldots(2j(1{-}z)-\frac{2j(2j{-}1)}{2!}(1{-}z)^2+\cdots)=\\&=(2\cdot 4 \cdots 2j (1{-}z)^j +(1+3+5 \cdots+2j-1) 2^{j{-}1}j!(1{-}z)^{j+1}+ \cdots)=\\&=2^j j! (1-z)^j-2^{j-1} j!\, j^2 (1-z)^{j+1}+\cdots
\end{align*}

It follows that 
\begin{gather*}
(z^2,z^2)_j (z^2,z^2)_{d-j}=(2^j j! (1-z)^j-2^{j-1} j!\, j^2 (1-z)^{j+1}+\cdots) \times \\ \times (2^{d-j} (d-j)! (1-z)^{d-j}-2^{d-j-1} (d-j)!\, (d-j)^2 (1-z)^{d-j+1}+\cdots)=\\=2^d j! (d-j)! (1-z)^d-2^{d-1} j! (d-j)! ((d-j)^2+j^2)(1-z)^{d+1}+\cdots
\end{gather*}

$(ii)$  To find the first terms of  the Laurent series for the function 
$$
\frac{(-1)^j z^{j(j+1)} (1+z)}{(z^2,z^2)_j\,(z^2,z^2)_{d-j}}. 
$$
we expand the numerator  in the  Taylor series by  $(1-z)$. We have
\begin{align*}
&1+z=2-(1-z),\\
&z^{j(j+1)}=1-j(j+1)(1-z)+\cdots,\\
&(1+z)z^{j(j+1)}=2-(2j(j+1)+1)(1-z)+\cdots
\end{align*}

It is easy to check that  the following decomposition holds: 

$$
\frac{a_0+a_1 x+ \cdots}{b_0+b_1x+\cdots}=\frac{a_0}{b_0}+\frac{a_1b_0-a_0b_1}{b_0^2}x+ \cdots, b_0 \neq 0.
$$

Then 
\begin{gather*}
\frac{(-1)^j z^{j(j+1)} (1+z)}{(z^2,z^2)_j\,(z^2,z^2)_{d-j}}=\frac{2-(2j(j+1)+1)(1-z)+\cdots}{2^d j! (d-j)! (1-z)^d-2^{d-1} j! (d-j)! ((d-j)^2+j^2)(1-z)^{d+1}+\cdots}=\\=
\frac{1}{(1-z)^d}\,\frac{2-(2j(j+1)+1)(1-z)+\cdots}{2^d j! (d-j)! -2^{d-1} j! (d-j)! ((d-j)^2+j^2)(1-z)+\cdots}=\\
=\frac{1}{(1-z)^d} \, \left(  \frac{1}{2^{d-1} j! (d-j)!} +\frac{(-1)^j}{2^{d-1} j! (d-j)!}(d+1)\left(\frac{1}{2}d-j-\frac 12 \right)(1-z)+\cdots \right). 
\end{gather*}
\end{proof}

The following lemma shows how the function  $ \varphi_n$ acts on the negative powers of $1-z.$
\begin{lm}
$$
 \varphi_n \left( \dfrac{1}{(1-z)^{h}} \right)=\sum_{i=0}^k \frac{\alpha_{n\,i}}{(1-z)^{i}}, h \in \mathbb{N}.
 $$
here $\alpha_{nh}=n^{h-1}, $ $\alpha_{n,h-1}=-n^{n-2}(n-1) \dfrac{h}{2}.$
\end{lm}
\begin{proof}
Using {Lemma 4} of the  article \cite{BB}, we get
$$
 \varphi_n \left( \dfrac{1}{(1-z)^{h}} \right)=\dfrac{\varphi_n \left( (1+z+z^2+\cdots+z^{n-1})^h\right)}{(1-z)^{h}}.$$
  Obviously $\alpha_{nh}$ is the remainder after  the division of $\varphi_n \left( (1+z+z^2+\cdots+z^{n-1})^h\right)$ by $(1-z).$

Using the definition of the function  $\varphi_n$ we get
\begin{gather*}
\varphi_n \left( (1+z+z^2+\cdots+z^{n-1})^h\right)=\frac{1}{n}\sum_{j=0}^{n-1} \left(1+\zeta_n^j z+(\zeta_n^j)^2 z^2+\cdots +(\zeta_n^j)^{n-1}z^{(n-1)}\right)^h \Bigl |_{z^n=z}.
\end{gather*}

The remainder of  division of this polynomial by  $(1-z)$ is equal to its value at the point $z=1$.  Thus
$$
\varphi_n \left( (1+z+z^2+\cdots+z^{n-1})^h\right)\Bigl |_{z=1}=\frac{1}{n}\sum_{j=0}^{n-1} \left(1+\zeta_n^j +(\zeta_n^j)^2 +\cdots +(\zeta_n^j)^{n-1}\right)^h=n^{h-1}.
$$

Obviously $\alpha_{n,h-1}$ is the coefficient of  $(1-z)$ in Taylor series expansion for $\varphi_n \left( (1+z+z^2+\cdots+z^{n-1})^h\right)$ at the point $z=1.$ Therefore
\begin{gather*}
\alpha_{n,h-1}=-\lim_{z \to 1}(\varphi_n \left( (1+z+z^2+\cdots+z^{n-1})^h\right))'.
\end{gather*}
We have
\begin{gather*}
\left(\frac{1}{n}\sum_{j=0}^{n-1} \left(1+\zeta_n^j z+(\zeta_n^j)^2 z^2+\cdots +(\zeta_n^j)^{n-1}z^{(n-1)}\right)^h  \right)'=\\
=\frac{h}{n}\sum_{j=0}^{n-1}  (1+\zeta_n^j z+(\zeta_n^j)^2 z^2+\cdots +(\zeta_n^j)^{n-1}z^{(n-1)})^{h-1} (\zeta_n^j +2(\zeta_n^j)^2 z+\cdots +(n-1)(\zeta_n^j)^{n-1}z^{(n-2)}). 
\end{gather*}
It now follows that 
\begin{gather*}
\lim_{z \to 1}\left(\frac{1}{n}\sum_{j=0}^{n-1} \left(1+\zeta_n^j z+(\zeta_n^j)^2 z^2+\cdots +(\zeta_n^j)^{n-1}z^{(n-1)}\right)^h  \right)'=\\
=\frac{h}{n}\sum_{j=0}^{n-1}  (1+\zeta_n^j +(\zeta_n^j)^2 +\cdots +(\zeta_n^j)^{n-1})^{h-1} (\zeta_n^j +2(\zeta_n^j)^2 +\cdots +(n-1)(\zeta_n^j)^{n-1})=\\
=\frac{h}{n} n^{h-1} (1+2+\ldots +(n-1))=\frac{1}{2} h(n-1)n^{h-1}.
\end{gather*}
 By using the relation $$
 \lim_{z \to 1}\left( f(z^n)|_{z^n=z}\right)'=\frac{1}{n}\lim_{z \to 1}f'(z^n),
 $$  we get
\begin{gather*}
\alpha_{n,h-1}= -\lim_{z \to 1}(\varphi_n \left( (1+z+z^2+\cdots+z^{n-1})^h\right))'=\\=
-\frac{1}{n}\lim_{z \to 1}\left(\frac{1}{n}\sum_{j=0}^{n-1} \left(1+\zeta_n^j z+(\zeta_n^j)^2 z^2+\cdots +(\zeta_n^j)^{n-1}z^{(n-1)}\right)^h  \right)'=-\frac{1}{2} h(n-1)n^{h-2}.
\end{gather*}
\end{proof}
Now we can compute  $\deg(\mathcal{C}_d), \psi(\mathcal{C}_d)$.
\begin{te}
\begin{align*}
&\deg(\mathcal{C}_d)=\lim_{z \to 1} (1-z)^d \mathcal{P}(\mathcal{C}_d,z)=\frac{1}{d!}\sum_{0 \leq j  < d/2}(-1)^j { d \choose j} \left(\frac{d}{2}-j \right)^{d-1},\\
&\psi(\mathcal{C}_d)=\lim_{z \to 1}\left( -(1-z)^d \mathcal{P}(\mathcal{C}_d,z)\right)'_z=\frac{1}{2}\deg(\mathcal{C}_d).
\end{align*}
\end{te}
\begin{proof}
Using {Lemma 1} and {Lemma 2} we get
\begin{gather*}
\mathcal{P}(\mathcal{C}_d,z)=\sum_{0 \leq j  < d/2}\varphi_{ d-2j}\left(\frac{(-1)^j z^{j(j+1)} (1+z)}{(z^2,z^2)_j\,(z^2,z^2)_{d-j}}\right)=\sum_{0 \leq j  < d/2}\varphi_{ d-2j}\left(  \frac{(-1)^j}{2^{d-1} j! (d-j)!}\frac{1}{(1-z)^{d}}+\cdots  \right)=\\=
\sum_{0 \leq j  < d/2}\frac{(-1)^j}{2^{d-1} j! (d-j)!} \varphi_{ d-2j}\left(  \frac{1}{(1-z)^{d}}  \right)+\\+\sum_{0 \leq j  < d/2}\frac{(-1)^j}{2^{d-1} j! (d-j)!}(d+1)\left(\frac{1}{2}d-j-\frac 12 \right)\varphi_{ d-2j}\left(\frac{1}{(1-z)^{d-1}} \right)+\cdots=\\=
\frac{1}{(1-z)^{d}} \sum_{0 \leq j  < d/2}\frac{(-1)^j (d-2j)^{d-1}}{2^{d-1} j! (d-j)!}-\frac{1}{(1-z)^{d-1}} \frac{1}{2}\sum_{0 \leq j  < d/2}\frac{(-1)^j}{2^{d-1} j! (d-j)!}(d-2j)^{d-2}(d-2j-1)(d-1)+\\
+\frac{1}{(1-z)^{d-1}} \frac{1}{2}\sum_{0 \leq j  < d/2}\frac{(-1)^j}{2^{d-1} j! (d-j)!}(d+1)(d-2j-1)(d-2j)^{d-2}+\cdots
\end{gather*}
Thus the coefficient of $\dfrac{1}{(1-z)^d}$ is 
\begin{gather*}
\deg(\mathcal{C}_d)=\sum_{0 \leq j  < d/2}\frac{(-1)^j (d-2j)^{d-1}}{2^{d-1} j! (d-j)!}=\frac{1}{d!}\sum_{0 \leq j  < d/2}(-1)^j { d \choose j} \left(\frac{d}{2}-j \right)^{d-1},
\end{gather*}

and the coefficient of $\dfrac{1}{(1-z)^{d-1}}$ is 

\begin{gather*}
\psi(\mathcal{C}_d)=\frac{1}{2 d!}\sum_{0 \leq j  < d/2}(-1)^j { d \choose j} \left(\frac{d}{2}-j \right)^{d-1}.
\end{gather*}
\end{proof}
The algebra of invariants  $\mathcal{C}_d$ is Gorenstein. Then its Poincare series $\mathcal{P}(\mathcal{C}_d,z)$ satisfies the following functional equation:
$$
\mathcal{P}(\mathcal{C}_d,z^{-1})=(-1)^d z^q \mathcal{P}(\mathcal{C}_d,z),
$$
here $q$  is the difference between  the degrees of the numerator and denominator of the   rational function $\mathcal{P}(\mathcal{C}_d,z),$  (see \cite{St}).
Using \cite[Theorem 2]{Pop}, we obtain  $$\frac{\deg(\mathcal{C}_d)}{\psi(\mathcal{C}_d)}=q-d.$$
Using {Theorem 2} we get $q=d+1=\dim V_d.$
 
It follows that the algebra of covariants  $\mathcal{C}_d$ is free algebra only if $d=1.$ Indeed, if $\mathcal{C}_d$ is free, then $q$ is equal to the sum of degrees  of  the generating elements. It is easily to verified that all powers of elements (except for possibly one) in  a minimal generating   system of the algebra $\mathcal{C}_d$ is greater than  2. Therefore  $q \leq 2+3(d-1).$ Thus $d+1 \leq 2+3(d-1)$ and $d \leq 1.$

%=====================================================================================

\section{Asymptotic behaviour of  $\deg(\mathcal{C}_d) $.}

%=====================================================================================

Let us establish an  integral representation  for the degree $\deg(\mathcal{C}_d)$. We denote by
$$c_d:=\deg(\mathcal{C}_d) \cdot \, d!=\sum_{0 \leq j  < d/2} (-1)^j { d \choose j}\left( \frac{d}{2}-j \right)^{d-1}.$$

The following statement holds:
\begin{lm}

$$
\begin{array}{ll}
%(i) & c_d=\displaystyle \frac{1}{2}\sum_{j=0}^d (-1)^j { d \choose j}{\rm sign} \left( \frac{d}{2}-j \right)\left( \frac{d}{2}-j \right)^{d-1},\\

(i) & c_d=2\pi^{-1}(d-1)!\displaystyle \int \limits _{0}^{\infty}\frac{{\sin ^d x}}{x^d}  dx, \\

(ii) & \deg(\mathcal{C}_d)>0.
\end{array}
$$
\end{lm}
\begin{proof} $(i)$ We have 
\begin{gather*}
2 c_d=\sum_{0 \leq j  < d/2} (-1)^j { d \choose j}\left( \frac{d}{2}-j \right)^{d-1}+\sum_{0 \leq j  < d/2} (-1)^j { d \choose j}\left( \frac{d}{2}-j \right)^{d-1}=\\=
\sum_{0 \leq j  < d/2} (-1)^j { d \choose j}\left( \frac{d}{2}-j \right)^{d-1}+\sum_{d/2 \leq j  < d} (-1)^j { d \choose j} {\rm sign}\left( \frac{d}{2}-j \right) \left( \frac{d}{2}-j \right)^{d-1}=\\
=\sum_{j=0}^d (-1)^j { d \choose j}{\rm sign} \left( \frac{d}{2}-j \right)\left( \frac{d}{2}-j \right)^{d-1}.
\end{gather*}

 We use that 
$$ \frac{\pi}{2}\, {\rm sign}(a)=\displaystyle \int \limits _{0}^{\infty}\frac{{\sin a x}}{x}  dx
$$
	Then 
	
\begin{gather*}
 \pi c_d=\frac{\pi}{2}\sum_{j=0}^d (-1)^j { d \choose j}{\rm sign} \left( \frac{d}{2}-j \right)\left( \frac{d}{2}-j \right)^{d-1}=\\=\sum_{j=0}^d (-1)^e { d \choose e}\left( \frac{d}{2}-e \right)^{d-1}\displaystyle \int \limits _{0}^{\infty}\frac{{\sin (\frac{d}{2}-j) x}}{x}  dx=\\=\displaystyle \int \limits _{0}^{\infty} {\rm Im} \left(\sum_{j=0}^d (-1)^j { d \choose j}\left( \frac{d}{2}-j \right)^{d-1}e^{i( \frac{d}{2}-j )x}\right)\frac{dx}{x}, i^2=-1.
\end{gather*}
We follows by the same method as   in \cite{SP}, Lemma 3.4.7.
We have 

\begin{gather*}
\sin ^d\frac{x}{2}=\left(\frac{e^{\frac{ix}{2}}-e^{-\frac{ix}{2}}}{2i}\right)^d=
	\frac{1}{2^d i^d}\sum_{j=0}^d  { d \choose j}\left(e^{\frac{ix}{2}}\right)^{d-j}\left(e^{-\frac{ix}{2}}\right)^j=
			\frac{1}{2^d i^d}\sum_{j=0}^ d (-1)^j { d \choose e}e^{ix(\frac{d}{2}-j)}.
	\end{gather*}
Differentiating $d-1$ times with respect to $x,$ we obtain
$$
\left(\sin ^d\frac{x}{2}\right)^{(d-1)}= \frac{i^{d-1}}{2^d i^d}\sum_{0 \leq e  \leq d} (-1)^e { d \choose e} \left(\frac{d}{2}-j\right)^{d-1}e^{ix(\frac{d}{2}-j)}.
$$
Hence 
$$
{\rm Im} \left(\sum_{j=0}^d (-1)^j { d \choose j}\left( \frac{d}{2}-j \right)^{d-1}e^{i( \frac{d}{2}-j )x}\right)=2^d  \left(\sin ^d\frac{x}{2}\right)^{(d-1)}.
$$

	Thus 
	
$$ c_d= \frac{1}{\pi } \int \limits _{0}^{\infty}2^d\left(\sin ^d\frac{x}{2}\right)^{(d-1)}\frac{dx}{x}=
						\frac{2}{\pi } \int \limits _{0}^{\infty}\left(\sin ^d x\right)^{(d-1)}\frac{dx}{x}
$$
	Integrating by parts $d-1$ times, we obtain
	$$ c_d=\frac{2(d-1)!}{\pi} \displaystyle \int \limits _{0}^{\infty}\frac{{\sin ^d x}}{x^d}  dx, $$

	$(ii)$ It is enougth to prove that $\displaystyle \int \limits _{0}^{\infty}\frac{{\sin ^d x}}{x^d}  dx>0. $  
 First of all we prove  that the  integral is  absolutely convergent. Let us split the integral into two parts:
$$\displaystyle \int \limits _{0}^{\infty}\frac{{\sin ^d x}}{x^d}  dx= \displaystyle \int \limits _{0}^{1}\frac{{\sin ^d x}}{x^d}  dx+\displaystyle \int \limits _{1}^{\infty}\frac{{\sin ^d x}}{x^d}  dx$$
Since  $\lim\limits_ {x\rightarrow 0}{\frac{\sin x}{x}}=1,$ then  the function  $\left(\frac{\sin x}{x}\right)^p$ is continuous on   $[0,1]$. Thus the first integral   is integrated on  $[0,1]$. Since 
$\left|\frac{{\sin ^d x}}{x^d}\right|  \leq \left|\frac{1}{x^d}\right|,$ then the second   integral is  absolutely convergent  for    $d>1$.

Now the integral can be represented in the form 
\begin{gather*}
\displaystyle \int \limits _{0}^{\infty}\frac{{\sin ^d x}}{x^d}  dx= \sum_{j=0} ^{\infty}{\left(\displaystyle \int \limits _{2 k\pi}^{(2k+1)\pi}\frac{{\sin ^d x}}{x^d}  dx+\displaystyle \int \limits _{(2k+1)\pi }^{4 k\pi}\frac{{\sin ^d x}}{x^d}  dx\right)}=\sum_{j=0} ^{\infty}{\left(\displaystyle \int \limits _{2 k\pi}^{(2k+1)\pi}\frac{{\sin ^d x}}{x^d}  dx+\displaystyle \int \limits _{2 k\pi}^{(2k+1)\pi}\frac{{\sin ^d (x+\pi)}}{(x+\pi)^d}  dx\right)}=\\=\sum_{j=0} ^{\infty}{\displaystyle \int \limits _{2 k\pi}^{(2k+1)\pi}\left(\frac{{\sin ^d x}}{x^d}  -\frac{{\sin ^d x}}{(x+\pi)^d}  \right) dx}=\sum_{j=0} ^{\infty}{\displaystyle \int \limits _{2 k\pi}^{(2k+1)\pi}\frac{{\sin ^d x}}{x^d (x+\pi)^d} \left((x+\pi)^d  -x^d\right) dx}>0,  d>1.
\end {gather*}
For the case $d=1$  we  have 
$$
\int \limits _{0}^{\infty}\frac{{\sin  x}}{x}  dx=\frac{\pi}{2}>0.
$$

\end {proof}

%????????? 
%$$
%u_{m,n}=\int_0^\infty \frac{\sin^m \,x}{x^n} dx,
%$$
%?? $m,n$ ???? ????? ????????? ???????? ??? ?????? ?? 2.

%? \cite{Ed}[?????? 1023  ??? ?? 333 (17)] ???????? ?????????? ????????????? 
%$$
%(n-1)(n-2) u_{m,n}+m^2 u_{m,n-2}-m(m-1)u_{m-2,n-2}=0, u_{1,1}=u_{2,2}=\frac{\pi}{2}.
%$$
%???? $n=m=2k$, ?? 
%$$
%u_{2k,2k}=\int_0^\infty \frac{\sin^{2k} \,x}{x^{2k}} dx=\frac{\pi}{2(2n-1)!} \left \langle  \begin{array}{l} 2k-1 \\ k-1 \end{array}  \right \rangle,
%$$
%?? 
%$$
%\left \langle  \begin{array}{l} n \\ k \end{array}  \right \rangle=\sum_{}
%$$

The condition $\deg(\mathcal{C}_d)>0$  is equivalent to that transcendence degree field of fractions of the algebra  $\mathcal{C}_d$ is equal to $d.$

Interestingly, in the general case   the  Wolstenholme formula  holds:
\begin{gather*}
\int_0^\infty \frac{\sin^p \,x}{x^s}=\frac{(-1)^{\frac{p-s}{2}}}{(s-1)!} \frac{\pi}{2^p} \sum_{p-2j>0}(-1)^j { p \choose j} (p-2j)^{s-1},
\end{gather*}
if $p-s$ is even, see. \cite[Problem 1033]{Ed}.

%============================================================
Finally, we deal with the asymptotic behavior of $\deg(\mathcal{C}_d )$  as $d$ tends to infinity. It amounts to the same by the previous lemma, to determine  the asymptotic behavior $\int \limits _{0}^{\infty}\dfrac{{\sin ^d x}}{x^d}  dx. $ 
\begin{te}

$$
\lim_{d\to\infty}d^{\frac{1}{2}}\displaystyle \int^{\infty}_{0}\frac{{\sin ^d x}}{x^d} dx=\frac{(6\pi)^{\frac{1}{2}}}{2}
$$

\end {te}
\begin {proof}
Write $$I=\int^{\infty}_{0}\frac{{\sin ^d x}}{x^d} dx,$$

and  split the limit into two parts

$$I=\lim_{d\to\infty}d^{\frac{1}{2}}\displaystyle \int^{\infty}_{0}\frac{{\sin ^d x}}{x^d} dx=\lim_{d\to\infty}d^{\frac{1}{2}}\displaystyle \int^{\frac{\pi}{2}}_{0}\frac{{\sin ^d x}}{x^d} dx+\lim_{d\to\infty}d^{\frac{1}{2}}\displaystyle \int^{\infty}_{\frac{\pi}{2}}\frac{{\sin ^d x}}{x^d} dx.$$
Since
$$
\left|{\int^{\infty}_{\frac{\pi}{2}}\frac{{\sin ^d x}}{x^d} dx}\right| \leq\displaystyle \int^{\infty}_{\frac{\pi}{2}} x^{-d} dx=\left.\lim_{b\to\infty} \frac{x^{-d+1}}{1-d}\right|_{\frac{\pi}{2}}^b=\frac{1}{(d-1)x^{d-1}} \to  0,
$$
it follows that
$$ I=\lim_{d\to\infty}d^{\frac{1}{2}} \int^{\frac{\pi}{2}}_{0}\frac{{\sin ^d x}}{x^d} dx.$$
Fix $\varepsilon>0$, sufficiently small. Since  $\frac{\sin x}{x}$ is monotonically decreasing as $0\leq x \leq \frac {\pi}{2}$, it follows that 
$\frac{\sin x}{x}\leq \frac{\sin \varepsilon}{\varepsilon}=1-\frac{\varepsilon^2}{3!}+ \frac{\varepsilon^4}{5!}-\cdots$, as $\varepsilon\leq x \leq \frac {\pi}{2}.$ It readily follows that there exists a strictly positive constant $a$ such that
$$\int^{\frac{\pi}{2}}_{\varepsilon}\frac{{\sin ^d x}}{x^d} dx=O\left(e^{-a \cdot d\varepsilon^2}\right).$$
We have for $0\leq x \leq \varepsilon$ 
$$\left(\frac{\sin x}{x}\right)^d=\left(1-\frac{1}{6}x^2+O\left(\varepsilon^4\right)\right)^d=e^{-\frac{1}{6}x^2+O\left(\varepsilon^4\right)}.$$
Hence 
\begin{gather*}
\int^{\varepsilon}_{0}{\left(\frac{\sin  x}{x}\right)}^d dx=e^{O\left(d \varepsilon^4\right)}\int^{\varepsilon}_{0}\left(-\frac{1}{6}dx^2\right) dx=e^{O\left(d \varepsilon^4\right)}\int^{\varepsilon d^{\frac{1}{2}}}_{0}e^{-\frac{1}{6} x^2}dx
\end{gather*}
Now choose $\varepsilon=\frac{\ln d}{\sqrt{d}}$. Then the limit  reduces to the Euler-Poisson integral:
$$ I=\lim_{d\to\infty}d^{\frac{1}{2}}\displaystyle \int^{\frac{\pi}{2}}_{0}\frac{{\sin ^d x}}{x^d} dx=\sqrt{6}\int^{\infty}_{0} e^{-x^2}dx=\frac{\sqrt{6 \pi}}{2}.
$$
 
\end {proof}

Thus asymptotic behaviour  $\deg(\mathcal{C}_d)$ as $d \to \infty$  is follows:
$$
\deg(\mathcal{C}_d)=\frac{c_d}{d!}\sim  \sqrt{\frac{6 }{\pi}} \frac{1}{d^{\phantom{}^\frac{3}{2}}}.
$$

\end{document}